\documentclass[11pt,a4paper]{amsart}

\usepackage{amsmath,amssymb,enumerate,amsthm,graphicx}

\newcommand{\UP}{\blacktriangle}
\newcommand{\DOWN}{\blacktriangledown}

\theoremstyle{plain}
\newtheorem{theorem}{Theorem}[section]
\newtheorem{proposition}[theorem]{Proposition}
\newtheorem{lemma}[theorem]{Lemma}
\newtheorem{corollary}[theorem]{Corollary}

\theoremstyle{remark}
\newtheorem{remark}[theorem]{Remark}
\newtheorem{example}[theorem]{Example}

\begin{document}

\title[Representation of Nelson algebras by Rough Sets]{Representation of Nelson algebras by Rough Sets Determined by Quasiorders}

\author{Jouni J{\"a}rvinen}
\address{Jouni J{\"a}rvinen: University of Turku, FI-20014~Turku, Finland}
\email{jouni.kalervo.jarvinen@gmail.com}
\author{S{\'a}ndor Radeleczki}
\address{S{\'a}ndor Radeleczki: Institute of Mathematics \\ 
University of Miskolc \\ 3515~Miskolc-Egyetemv{\'a}ros \\ Hungary}
\email{matradi@uni-miskolc.hu}


\begin{abstract}
In this paper, we show that every quasiorder $R$ induces a Nelson 
algebra $\mathbb{RS}$ such that the underlying rough set lattice $RS$ 
is algebraic. We note that $\mathbb{RS}$ is a three-valued
{\L}ukasiewicz algebra if and only if $R$ is an equivalence.
Our main result says that if $\mathbb{A}$ is a Nelson algebra 
defined on an algebraic lattice, then there exists a set $U$
and a quasiorder $R$ on $U$ such that $\mathbb{A} \cong \mathbb{RS}$.
\end{abstract}

\maketitle

\section{Introduction} \label{Sec:Intro}

Nelson algebras, also called $\mathcal{N}$-lattices or quasi-pseudo-Boolean
algebras, were introduced by H. Rasiowa as algebraic counterparts of the
constructive logic with strong negation by D. Nelson and A. A. Markov (see \cite{Rasiowa74}). 
They can be considered also as a generalisation of Boolean algebras. 
It is well known that any Boolean algebra defined on an algebraic lattice
is isomorphic to the powerset algebra $\wp(U)$ of some set $U$. 
In this paper, we prove an analogous result for Nelson algebras with 
algebraic underlying lattices and algebras of rough sets determined
by quasiorders.

Rough sets were introduced by Z.~Pawlak in \cite{Pawl82}.
In rough set theory it is assumed that our knowledge
about a universe of discourse $U$ is given in terms of
a binary relation reflecting the distinguishability or 
indistinguishability of the elements of $U$. Originally,
Pawlak assumed that this binary relation is an equivalence,
but in the literature numerous studies can be found 
in which approximations are determined also by other types of relations.

If $R$ is a given binary relation on $U$, then for any 
subset $X \subseteq U$,  the  \emph{lower approximation} of $X$ is defined as
\[
 X^\DOWN = \{x \in U \mid R(x) \subseteq X \}
\]
and the \emph{upper approximation of $X$} is 
\[
X^\UP = \{x \in U \mid R(x) \cap X \neq \emptyset \},
\] 
where $R(x) = \{ y \in U \mid x \, R \, y \}$.
The \emph{rough set of $X$} is the pair $\mathcal{A}(X) = (X^\DOWN,X^\UP)$
and the set of all \emph{rough sets} is
\[
 RS = \{ \mathcal{A}(X)  \mid X \subseteq U \}.
\]
The set $RS$  may be canonically ordered by the
coordinatewise order: 
$\mathcal{A}(X) \leq  \mathcal{A}(Y)$
holds in $RS$ if $X^\DOWN \subseteq Y^\DOWN$ and  $X^\UP \subseteq Y^\UP$. 

The structure of $RS$ is well studied in the case when $R$ is an equivalence;
see \cite{Com93, DemOrl02, GeWa92, Jarv07,Pagliani97,PagChak08,PomPom88}. 
In particular, J.~Pomyka{\l }a and J.~A.~Pomyka{\l }a showed in \cite{PomPom88} that 
$RS$ is a Stone lattice. Later this result was improved by S.~D.~Comer \cite{Com93} 
by showing that $RS$ is a regular double Stone lattice.
In  \cite{GeWa92}, M.~Gehrke and E.~Walker proved that $RS$ is isomorphic to 
$\mathbf{2}^I \times \mathbf{3}^J$, where $I$ is the set of singleton 
$R$-classes and $J$ is the set of non-singleton 
equivalence classes of $R$. Additionally, $RS$ forms a 
three-valued {\L}ukasiewicz algebra, as shown by P.~Pagliani \cite{Pagliani97}.
If $R$ is reflexive and symmetric or just transitive, then $RS$ is not necessarily 
even a semilattice. If $R$ is symmetric and transitive, then the structure of 
$RS$ is as in case of equivalences \cite{Jarv04}.

In \cite{JRV09}, we proved that any $RS$ determined by a quasiorder $R$
is a completely distributive lattice isomorphic to a complete 
ring of sets, and we described its completely join-irreducible elements. 
We also showed that $\mathbb{RS} = (RS,\cup, \cap, c,(\emptyset,\emptyset), (U,U))$
is a De~Morgan algebra, where the operation $c$ is defined by
$c \colon \mathcal{A}(X) \mapsto \mathcal{A}(U \setminus X)$.
In this paper, we prove that $\mathbb{RS}$ is in fact a Nelson
algebra defined on an algebraic lattice. The main objective of 
this work is to prove the following representation theorem.

\begin{theorem} \label{Thm:MAIN}
Let\/ $\mathbb{A} = (A,\vee,\wedge,c,0,1)$ be a Nelson algebra defined on an algebraic lattice. 
Then, there exists a set $U$ and a quasiorder $R$ on $U$ such that $\mathbb{A} \cong \mathbb{RS}$.
\end{theorem}

As a corollary we can also show that if $\mathbb{A}$ is a semisimple
Nelson algebra with an underlying algebraic lattice, then there exists 
a set $U$ and an equivalence $R$ on $U$ such that $\mathbb{A} \cong \mathbb{RS}$.

The paper is structured as follows. In the next section we recall some
notions and facts related to De~Morgan, Kleene, Nelson, and Heyting algebras. 
Section~\ref{SEC:CompletelyDistributiveKleene+Nelson} summarises 
some more or less known properties of completely join-irreducible elements 
of completely distributive Kleene algebras, which will be used in the proofs of our 
main results. In Section~\ref{Sec:RoughSetNelson}, 
we prove that rough set lattices induced by quasiorders determine Nelson algebras. 
We also show that $\mathbb{RS}$ is a three-valued {\L}ukasiewicz 
algebra only in case $R$ is an equivalence. Section~\ref{Sec:Representation} 
contains the proof of Theorem~\ref{Thm:MAIN} and some of its consequences.

\section{Preliminaries} \label{Sec:Preliminaries}

Systematic treatments of De~Morgan and Kleene algebras can be found in \cite{BaDw74,Rasiowa74}.
A \emph{De~Morgan algebra} $\mathbb{A}=(A,\vee,\wedge,c,0,1)$ is an algebra
of type $(2,2,1,0,0)$ such that $A$ is a bounded distributive
lattice with a least element $0$ and a greatest element $1$, and $c$ is a
unary operation that satisfies for all $x,y \in A$,
\begin{center}
$c(c(x)) = x$;\\
 $x\leq y$ \ if and only if \ $c(x)\geq c(y)$.
\end{center}
This definition implies that $c$ is an isomorphism between
the lattice $A$ and its dual $A^\partial$. Hence, it satisfies the equations:
\begin{eqnarray*}
c(x \vee y) & = & c(x) \wedge c(y); \\
c(x \wedge y) & = & c(x) \vee c(y). 
\end{eqnarray*}

An element $x$ of a complete lattice $L$ is \emph{completely join-irreducible}
if for every subset $S$ of $L$, $x = \bigvee S$ implies that $x \in S$.
The set of completely join-irreducible elements of $L$ is
denoted by $\mathcal{J}(L)$ --- or simply by $\mathcal{J}$ when there is no danger
of confusion. For any $x$, let $J(x) = \{ j \in \mathcal{J} \mid j \leq x \}$.

A complete lattice $L$ is \emph{completely distributive} if for any doubly indexed
family of elements $\{x_{i,\,j}\}_{i \in I, \, j \in J}$ of $L$, we have
\[
\bigwedge_{i \in I} \Big ( \bigvee_{j \in J} x_{i,\,j} \Big ) = 
\bigvee_{ f \colon I \to J} \Big ( \bigwedge_{i \in I} x_{i, \, f(i) } \Big ), \]
that is, any meet of joins may be converted into the join of all
possible elements obtained by taking the meet over $i \in I$ of
elements $x_{i,\,k}$\/, where $k$ depends on $i$.

We say that the De~Morgan algebra $\mathbb{A}$ is \textit{completely distributive}, 
if its underlying lattice $A$ is completely distributive. In such a case, we 
may define for any $j \in \mathcal{J}$ the element
\[ j^* = \bigwedge \{ x \in A \mid x \nleq c(j) \}. \]
It is well known that $j^{\ast}\in \mathcal{J}$ (see e.g. \cite{Mont63a}). 
The next lemma for a finite $\mathbb{A}$ was proved in \cite{Mont63a}, 
and it is contained implicitly in \cite{Cign86}.

\begin{lemma} \label{Lem:StarProperties}
If\/ $\mathbb{A}$ is a completely distributive De~Morgan algebra,
then for all $i,j \in \mathcal{J}$:
\begin{enumerate}[\rm (a)]
\item $j^* \nleq c(j)$;
\item $i \leq j$ implies $i^* \geq j^*$;
\item $j^{**} = j$.
\end{enumerate}
\end{lemma}
Notice that statements (b) and (c) of Lemma~\ref{Lem:StarProperties} 
mean that the map $j \mapsto j^*$ is an order-isomorphism between
the ordered set $\mathcal{J}$ and its dual $\mathcal{J}^\partial$.

A complete lattice $L$ is said to be \emph{algebraic} 
if any element $x\in L$ is the join of a set of
compact elements of $L$ (see e.g. \cite{Grat98}). 
A \emph{complete ring of sets} is a family of sets $\mathcal{F}$ such that
$\bigcup \mathcal{H}$ and $\bigcap \mathcal{H}$ belong to $\mathcal{F}$ 
for any $\mathcal{H} \subseteq  \mathcal{F}$.

In the next remark we give some conditions under which
a lattice is isomorphic to a complete ring of sets (cf. \cite{DaPr02}).

\begin{remark}\label{Rem:CharacterazaionLatticesOfSets}
Let $L$ be a lattice. The following are equivalent:
\begin{enumerate}[\rm (a)]
\item $L$ is isomorphic to a complete ring of sets;
\item $L$ is algebraic and completely distributive;
\item $L$ is distributive and doubly algebraic (i.e.\@ both $L$ and $L^\partial$ are algebraic);
\item $L$ is algebraic, distributive and every element of $L$ is a join of completely join-irreducible elements of $L$.
\end{enumerate}
\end{remark}

If $\mathbb{A}$ is a De~Morgan algebra whose underlying lattice is algebraic, 
then $A$ is doubly algebraic, since $A$ is self-dual. Thus, $A$ has all
equivalent properties of Remark~\ref{Rem:CharacterazaionLatticesOfSets}. 
The next connection between the maps $c \colon A \rightarrow A$ and 
$^* \colon \mathcal{J} \rightarrow \mathcal{J}$
was proved in \cite{Mont63a} for finite algebras and in the
completely distributive case it is contained implicitly
in \cite{Vaka77} (see also \cite{Varl85}).

\begin{lemma} \label{Lem:StarCompl}
If\/ $\mathbb{A}$ is a De~Morgan algebra defined on an algebraic lattice,
then for all $x \in A$,
$c(x) = \bigvee \{ j \in \mathcal{J} \mid j^* \nleq x\}$.
\end{lemma}

Let $L$ and $K$ be two completely distributive lattices such that
any element of them is a join of completely join-irreducible elements, and
assume that $\varphi\colon \mathcal{J}(L) \to \mathcal{J}(K)$ is an order-isomorphism of ordered sets. 
G.~Birkhoff \cite{Birk95} proved that in this case the map
$\Phi \colon L \to K$,
\[
\Phi(x) = \bigvee \varphi(J(x)) 
\]
is a lattice-isomorphism. One may extend this result to De~Morgan algebras.

\begin{corollary}\label{Cor:DemorgranIsom}
Let\/ $\mathbb{L} = (L,\vee,\wedge,c,0,1)$ and $\mathbb{K} = (K,\vee,\wedge,c,0,1)$ 
be two De~Morgan algebras defined on algebraic lattices.
If $\varphi \colon \mathcal{J}(L) \to \mathcal{J}(K)$ is an  order-isomorphism 
such that
\[
\varphi(j^*) = \varphi(j)^*
\]
for all $j \in \mathcal{J}(L)$, then $\Phi$ is an isomorphism between the algebras 
$\mathbb{L}$ and $\mathbb{K}$.
\end{corollary}

\begin{proof}
Since the map $\Phi \colon L \to K$ is 
a lattice-isomorphism that maps $J(x)$ onto $J(\Phi(x))$, the set 
$\{ j \in \mathcal{J}(L) \mid j \nleq x\}$ is mapped by $\Phi$ onto the set 
$\{ k\in \mathcal{J}(K)\mid k \nleq\Phi(x) \}$. By using this fact, we prove that
\[
\Phi(c(x))=c(\Phi(x)).
\]
Indeed, since $j \in \mathcal{J}(L)$, $\Phi(j) = k$ implies 
\[ \Phi(j^*)=\varphi(j^*) = \varphi(j)^* = \Phi(j)^* = k^*,\] 
the set $\{ j^* \mid j \in \mathcal{J}(L) \mbox{ and } j \nleq x\}$ 
is mapped by $\Phi$ onto the set  
$\{k^* \mid k \in \mathcal{J}(K) \mbox{ and } k \nleq\Phi(x)\}$. 
By Lemma~\ref{Lem:StarCompl}, 
\[ c(x) = \bigvee \{ j \in \mathcal{J}(L) \mid j^* \nleq x\} 
= \bigvee \{ j^* \mid j \in \mathcal{J}(L) \mbox{ and } j \nleq x\};\] 
recall that $j = j^{**}$. Similarly, we get
\[ 
c(\Phi(x)) = \bigvee \{ k^* \mid k \in \mathcal{J}(K) \mbox{ and } k \nleq \Phi(x) \}.
\]
Hence, we have 
\begin{align*}
\Phi(c(x)) &= \bigvee \{ \Phi(j^*) \mid j \in \mathcal{J}(L)\mbox{ and } j^* \nleq x\} \\
           &= \bigvee \{ k^* \mid k \in \mathcal{J}(K) \mbox{ and } k \nleq \Phi(x) \} \\
           &= c(\Phi(x)).
\end{align*}
\end{proof}

A De~Morgan algebra $\mathbb{A}$ is a \emph{Kleene algebra}
if for all $x,y \in A$,  
\[
x \wedge c(x) \leq y \vee c(y).
\]
We define for a Kleene algebra $\mathbb{A}$ two sets:
\[
 A^+ = \{x \vee c(x) \mid x \in A\} \quad \mbox { and } \quad  A^- = \{x \wedge c(x) \mid x \in A\}.
\]
The proof of the following lemma is straightforward and is omitted.

\begin{lemma} \label{Lem:KleeneVeeWedge}
Let\/  $\mathbb{A}$ be a Kleene algebra. Then,
\begin{enumerate}[\rm (a)]
\item $c(A^+) = A^-$ \ and \ $c(A^-) = A^+$;
\item $a \leq c(b)$ for all $a,b \in A^-$;
\item $c(a) \leq b$ for all $a,b \in A^+$; 
\item $a \in A^-$ \ iff \ $a \leq c(a)$;
\item $a \in A^+$ \ iff \ $c(a) \leq a$.
\end{enumerate}
\end{lemma}

Additionally, for any Kleene algebra $\mathbb{A}$, $A^{-}$ is an 
ideal and $A^{+}$ is a filter of $A$, as noted in  e.g. \cite{KaarliPixley}.
Let $\mathbb{A}$ be a Kleene algebra defined on a complete lattice $A$.
We denote
\[ \alpha = \bigvee A^- \quad \mbox{ and } \quad \beta = \bigwedge A^+.\]
Lemma~\ref{Lem:KleeneVeeWedge} implies easily the following result.

\begin{corollary} \label{Cor:Kleene3}
If\/ $\mathbb{A}$ is a Kleene algebra defined on a complete lattice,
then $A^- = (\alpha]$, $A^+ = [\beta)$, and $c(\alpha) = \beta$.
\end{corollary}

A \emph{Heyting algebra} $L$ is a bounded lattice such that for all $a,b \in L$, there 
is a greatest element $x$ of $L$ such that
\[ a \wedge x \leq b .\]
This element is the \emph{relative pseudocomplement} of $a$ with respect to $b$, 
and is denoted $a \Rightarrow b$. It is well known that any completely distributive
lattice $L$ is a Heyting algebra $(L,\vee,\wedge,\Rightarrow,0,1)$
such that the relative pseudocomplement is defined as 
\[
x \Rightarrow y  =  \bigvee \big \{ z \in L \mid z \wedge x \leq y \big \}.
\]

According to R. Cignoli [3], a \emph{quasi-Nelson} algebra is a Kleene algebra $(A,\vee,\wedge,c,0,1)$
such that for each pair $a$ and $b$ of its elements, the relative pseudocomplement
\[ a \Rightarrow (c(a) \vee b) \]
exists. This means that every Kleene algebra whose underlying lattice is a Heyting
algebra, and, in particular, any Kleene algebra defined on an algebraic lattice, 
forms a quasi-Nelson algebra.

In quasi-Nelson algebras, $a \Rightarrow (c(a) \vee b)$ is denoted simply by
$a \to b$ and this is called the \emph{weak relative pseudocomplement} of $a$
with respect to $b$. As shown by D.~Brignole and A.~Monteiro in \cite{BrigMont67}, 
the operation $\to$ satisfies the equations:
\begin{enumerate}[({N}1)]
\item $a \to a = 1$;
\item $(c(a) \vee b) \wedge (a \to b ) = c(a) \vee b$;
\item $a \wedge (a \to b) = a \wedge (c(a) \vee b)$;
\item $a \to (b \wedge c) = (a \to b) \wedge (a \to c)$.
\end{enumerate}

A \emph{Nelson algebra} is a quasi-Nelson algebra satisfying the equation
\begin{enumerate}[(N5)]
\item[(N5)] $(a\wedge b)\rightarrow c = a \rightarrow (b\rightarrow c)$.
\end{enumerate}
Note that it is shown in  \cite{BrigMont67} that Nelson algebras can be
equationally characterized as algebras $(A, \vee, \wedge, \to, c, 0, 1)$,
where $(A, \vee, \wedge, c, 0, 1)$ is a Kleene algebra and the binary 
operation $\to$ satisfies (N1)--(N5).

\section{Completely Distributive Kleene and Nelson Algebras}
\label{SEC:CompletelyDistributiveKleene+Nelson}

In this section, we present some more or less known properties
that are used in the proofs of our main results.
First, we show that in completely distributive Kleene algebras,
the set of completely join-irreducible elements $\mathcal{J}$ 
can be divided into three disjoint sets in terms of the map 
$^*\colon \mathcal{J} \to \mathcal{J}$.

For any set $A$, let $X'$ denote the set-theoretical complement $A \setminus X$  
of any subset $X \subseteq A$. Let $\mathbb{A}$ be a De~Morgan algebra.
We denote the set of its prime filters by $\mathcal{F}_p$.  
A map $g \colon \mathcal{F}_p \to \mathcal{F}_p$ is defined by setting
\[
g(P) =  c(P)'
\]
for all $P \in \mathcal{F}_p$. Since $c(P)$ is a prime ideal of the lattice $A$, 
$c(P)'$ is a prime filter of $A$. It is easy to verify (see e.g.\@ \cite{Cign86}) 
that for all $P \in \mathcal{F}_p$,
\[ g(g(P)) = P \]
and for any $P,Q \in \mathcal{F}_p$, we have
\[
P \subseteq Q \mbox{ \ implies \ } g(Q)\subseteq g(P).
\]

If $\mathbb{A}$ is a completely distributive De~Morgan algebra, 
the map $g$ is related to the map $^* \colon \mathcal{J}  \to \mathcal{J}$, because
for every $j \in \mathcal{J}$, 
\[
g([j)) = \{ c(x) \mid j \leq x\}' = \{x \mid x \leq c(j) \}'  = \{x \mid x \nleq c(j) \}  = [j^*);
\]
note that in any distributive lattice, $[j)$ is a prime filter for each $j \in \mathcal{J}$.

Let $\mathbb{A}$ be a Kleene algebra. It is known (see e.g.\@ \cite{Cign86})
that in this case for any prime filter $P$ of $A$, we have either $g(P) \subseteq P$
or $P \subseteq g(P)$. In \cite{Cign86}, the 
following two sets were defined:
\begin{align*}
\mathcal{F}_p^+ &= \{ P \in \mathcal{F}_p \mid P \subseteq g(P) \}; \\
\mathcal{F}_p^- &= \{ P \in \mathcal{F}_p \mid g(P) \subseteq P \}.
\end{align*}
Then, $\mathcal{F}_p = \mathcal{F}_p^+ \cup \mathcal{F}_p^-$.

\begin{remark}\label{Rem:Comparable}
From the above, it follows that for any $j \in \mathcal{J}$ in a completely
distributive Kleene algebra $\mathbb{A}$, $j$ and $j^*$ are comparable. 
Indeed, for any $j \in \mathcal{J}$, either $[j) \in \mathcal{F}_p^+$
or $[j) \in \mathcal{F}_p^-$ holds. If $[j) \in \mathcal{F}_p^+$, then 
$[j) \subseteq g([j)) = [j^*)$ implies $j^* \leq j$, and for 
$[j) \in \mathcal{F}_p^-$, $[j^*) = g([j)) \subseteq [j)$ gives
$j \leq j^*$.
\end{remark}

Let $\mathbb{A}$ be a completely distributive Kleene algebra. 
We may now define three disjoint sets:
\begin{align*}
\mathcal{J}^- & = \{ j \in \mathcal{J} \mid j < j^* \}; \\
\mathcal{J}^* & = \{ j \in \mathcal{J} \mid j = j^* \}; \\
\mathcal{J}^+ & = \{ j \in \mathcal{J} \mid j > j^* \}. 
\end{align*}
Then clearly,
\begin{center}
$[j) \in \mathcal{F}_p^+ \iff j \in \mathcal{J}^+ \cup \mathcal{J}^*$, \\
$[j) \in \mathcal{F}_p^- \iff j \in \mathcal{J}^- \cup \mathcal{J}^*$,
\end{center}
and in view of Remark~\ref{Rem:Comparable}, we have 
$\mathcal{J} = \mathcal{J}^- \cup  \mathcal{J}^* \cup  \mathcal{J}^+$.
The next lemma contains some simple known facts, 
but it is proved to make our proofs consistent. 

\begin{lemma}\label{Lem:KleeneDivision}
If\/ $\mathbb{A}$ is a completely distributive Kleene algebra, then for all $j \in \mathcal{J}$:

\begin{enumerate}[\rm (a)]
\item If $j \notin A^-$, then $j^* \leq j$;
\item $\mathcal{J}^- = \mathcal{J} \cap A^-$;
\item $j \in \mathcal{J}^-  \iff j^* \in \mathcal{J}^+$.
\end{enumerate}
\end{lemma}

\begin{proof} 
(a) If $j \notin A^-$, then $j \nleq c(j)$ by Lemma~\ref{Lem:KleeneVeeWedge}(d)
and so $j^* \leq j$.

(b) If $j \in \mathcal{J}^-$, then $j \notin A^-$ is not possible by (a),
and we obtain $\mathcal{J}^- \subseteq \mathcal{J} \cap A^-$. Conversely, 
if $j \in \mathcal{J} \cap A^-$, then $j \leq c(j)$ and hence 
$j^* \leq j$ is not possible. Thus, we get $j < j^*$, that is, 
$j \in \mathcal{J}^-$. This proves $\mathcal{J}^- = \mathcal{J} \cap A^-.$

(c) If $j \in \mathcal{J}^-$, then $j^* > j = j^{**}$ and $j^* \in \mathcal{J}^+$.
The other direction is proved analogously.
\end{proof}

Following A.~Monteiro \cite{Mont63a}, a Kleene algebra $\mathbb{A}$ is said 
to have the \emph{interpolation property}, if for any 
$P,Q \in \mathcal{F}_p^+$ such that $P \subseteq g(Q)$, 
there exists a prime filter $F$ of $A$ fulfilling the conditions 
$P \subseteq F \subseteq g(P)$ and $Q \subseteq F \subseteq g(Q)$ 
(see also \cite{Mont80}).

\begin{theorem}\cite[Theorem 3.5]{Cign86}  \label{Thm:Interpolation}
A quasi-Nelson algebra is a Nelson algebra if and only if it has the
interpolation property.
\end{theorem}

\begin{lemma}\cite[Lemma 2.2]{Cign86}\label{Lem:Cignoli} 
A Kleene algebra $\mathbb{A}$ has the interpolation property 
if and only if, for given $P,Q \in \mathcal{F}_p^+$ such that $P\subseteq g(Q)$,
\[ a \wedge b \nleq c(a) \vee c(b) \]
for all $a \in P$ and $b \in Q$.
\end{lemma}

We note that in case of a finite Kleene algebra, the interpolation property 
is equivalent to  condition:
\begin{tabbing}
(M) \ \=For any $p$ and $q$ in $\mathcal{J}$ such that $p^*,q^* \leq p,q$, 
there is $k$ in $\mathcal{J}$ such that 
\end{tabbing}
\[
p^*,q^* \leq k \leq p,q 
\]
(see \cite{Mont63a}). Moreover, the equivalence of these two conditions in 
case of Kleene algebras with an algebraic underlying lattice implicitly 
follows from \cite[Theorem~2]{Vaka77}. Since the approach and the terminology 
of \cite{Vaka77}  is notably different from ours, to avoid recalling
several notions from there and to make
our proofs self-consistent, below we present a direct proof.

\begin{proposition} \label{Prop:CharactNelson}
If\/ $\mathbb{A}$ is a Kleene algebra defined on an algebraic lattice, then
$\mathbb{A}$ has the interpolation property if and only if condition
{\rm (M)} is satisfied.
\end{proposition}

\begin{proof}($\Rightarrow$) Suppose that $\mathbb{A}$ has the interpolation property 
and that for some $p,q \in \mathcal{J}$, $p^*,q^* \leq p,q$ is satisfied. 
Then, for the prime filters $P = [p)$ and $Q = [q)$, we have
$g(P) = [p^*)$, $g(Q) = [q^*)$, and $P,Q \subseteq g(P),g(Q)$.
It is clear that $P,Q \in \mathcal{F}_p^+$, $p \in P$, and $q \in Q$.
Therefore, by Lemma~\ref{Lem:Cignoli},
\begin{equation}\label{EQ:difference} \tag{\dag}
 p \wedge q \nleq c(p) \vee c(q). 
\end{equation}
Because $A$ is an algebraic lattice and $\mathbb{A}$ is a De~Morgan algebra, each
element $x$ of $A$ is the join of completely join-irreducible
elements below $x$. Then \eqref{EQ:difference} implies that there exists 
a completely join-irreducible element $k \in \mathcal{J}$ such
that $k \leq p \wedge q$, but $k \nleq c(p) \vee c(q)$. Now clearly
$k \leq p,q$ and $k \nleq c(p),c(q)$, that is, $p^*,q^* \leq k$
and condition (M) holds.

\medskip\noindent%
($\Leftarrow$) Assume that condition (M) is satisfied, 
but $\mathbb{A}$ does not have the interpolation property. Then, by Lemma~\ref{Lem:Cignoli},
there are two prime filters $P$ and $Q$ satisfying $P \subseteq g(P)$,
$Q \subseteq g(Q)$, and $P \subseteq g(Q)$, as well as elements
$a \in P$ and $b \in Q$ such that
\[ a \wedge b \leq c(a) \vee c(b). \]
First, we show that for any $F \in \mathcal{F}_p^+$ 
and $x \in F$, the element
\[ 
 x^u = \bigvee \{ j \mid j \in J(x) \setminus \mathcal{J}^- \}
\]
belongs to $F$. For that, we also define the element
\[
 x^d = \bigvee \{ j \mid j \in J(x) \cap \mathcal{J}^- \}.
\]
Clearly, $x = x^u \vee x^d$. Since $x \in F$ and $F$ is a prime filter,
we have that $x^u \in F$ or $x^d \in F$. Now $x^d \in A^-$,
because $\mathcal{J}^- \subseteq A^-$ and $A^-$ 
is closed under any join, as it is a principal ideal of $A$, by
Corollary~\ref{Cor:Kleene3}. Thus, $x^d \leq c(x^d)$ by Lemma~\ref{Lem:KleeneVeeWedge}(d). 
If $x^d \in F$, then $c(x^d) \in F$ contradicting 
$c(x^d) \in c(F) = g(F)' \subseteq F'$. Therefore, we must
have $x^d \notin F$ and $x^u \in F$. 

Secondly, we prove that there exist $p \in J(a) \setminus \mathcal{J}^-$ and 
$q \in J(b) \setminus \mathcal{J}^-$ such that $p \nleq c(q)$.
If we assume that for all $p \in J(a) \setminus \mathcal{J}^-$ and 
$q \in J(b) \setminus \mathcal{J}^-$, $p \leq c(q)$ holds, then
\begin{align*}
a^u &= \bigvee \{ p \mid p \in J(a) \setminus \mathcal{J}^- \} \\
    &\leq \bigwedge \{ c(q) \mid q \in J(b) \setminus \mathcal{J}^- \} \\
    &= c \big ( \bigvee \{ q \mid q \in J(b) \setminus \mathcal{J}^- \} \big )\\
    &= c(b^u).
\end{align*}
As $P$ is a filter and $a^u \in P$, we obtain $c(b^u) \in P$, which contradicts
$c(b^u) \in c(Q) = g(Q)' \subseteq P'$. Thus, $p \nleq c(q)$ for
some $p \in J(a) \setminus \mathcal{J}^-$ and  $q \in J(b) \setminus \mathcal{J}^-$.

Finally,  we have $p^* \leq p$ and $q^* \leq q$, and $p \nleq c(q)$ implies $q^* \leq p$
by the definition of $q^*$. From this, by Lemma~\ref{Lem:StarProperties}, we also get 
$p^* \leq q^{**} = q$. By our original assumption, there exists an element $k \in \mathcal{J}$
such that 
\[ p^*,q^* \leq k \leq p,q .\]
Notice that this also directly implies $p^*,q^* \leq k^* \leq p,q $. 

The elements $a \in P$ and $b \in Q$ satisfy $a \wedge b \leq c(a) \vee c(b) = c(a \wedge b)$.
Because $p \leq a$ and $q \leq b$, 
\[ k,k^* \leq p \wedge q \leq a \wedge b \leq c(a \wedge b) \leq c(k),c(k^*).\]
This then means that both $k$ and $k^*$ are in $\mathcal{J}^-$.
But this is impossible by Lemma~\ref{Lem:KleeneDivision}(c). Therefore,
the interpolation property must hold.
\end{proof}

\section{ Nelson Algebras of Rough Sets Determined by Quasiorders} \label{Sec:RoughSetNelson}

In \cite{JRV09}, we proved that if $U$ is a non-empty set and $R$
is a quasiorder on $U$, then $RS$ is a complete sublattice of $\wp(U) \times \wp(U)$.
Since $\wp(U) \times \wp(U)$ is an algebraic, completely distributive lattice, this
implies that $RS$ is also an algebraic completely distributive lattice. Thus, $RS$
has the properties listed in Remark~\ref{Rem:CharacterazaionLatticesOfSets} and
\[
 \bigwedge_{i\in I} \mathcal{A}(X_{i}) = \Big ( \bigcap_{i\in I} X_{i}^\DOWN, \bigcap_{i\in I} X_{i}^\UP \Big )
\mbox{ \quad and \quad }
 \bigvee_{i\in I} \mathcal{A}(X_{i}) = \Big ( \bigcup_{i\in I} X_{i}^\DOWN, \bigcup_{i\in I} X_{i}^\UP \Big )
\]
for all $\{ \mathcal{A}(X_{i}) \mid i \in I \} \subseteq RS$. It is easy to observe
that $(\emptyset,\emptyset)$ is the least and $(U,U)$ is the greatest element of $RS$.
We also showed that the set of completely join-irreducible elements of $RS$ is
\begin{equation} \label{Eq:RS-JoinIrr}
  \mathcal{J} = \{(\emptyset,\{x\}^\UP) \mid \ 
|R(x)| \geq 2 \} \cup  \{( R(x), R(x)^\UP) \mid x \in U\}, \tag{$\star$}
\end{equation}
and that every element can be represented as a join of elements in $\mathcal{J}$.

In addition, we proved that the map
\[
c \colon RS \to RS, \, \mathcal{A}(X) \mapsto \mathcal{A}(X')
\]
is a De~Morgan complement, and therefore 
\[
 \mathbb{RS} = (RS,\vee,\wedge,c,(\emptyset,\emptyset), (U,U))
\]
is a De~Morgan algebra. Note that
\[
\mathcal{A}(X') = ((X^\prime)^\DOWN,(X^\prime)^\UP) = ((X^\UP)^\prime,(X^\DOWN)^\prime).
\]

Additionally, $\{x\}^\UP = \{y \in U \mid y \, R \, x \} = R^{-1}(x)$ for all $x \in U$.
Because $R$ is reflexive, $X^\DOWN \subseteq X \subseteq X^\UP$, and
transitivity means $X^{\UP \UP} \subseteq X^\UP$ and $X^\DOWN \subseteq X^{\DOWN \DOWN}$
for all $X \subseteq U$. In fact, $^\UP \colon \wp(U) \to \wp(U)$ is a topological 
closure operator, and $^\DOWN \colon \wp(U) \to \wp(U)$ is a topological interior operator;
see \cite{Jarv07}.

\begin{lemma} \label{Lem:RSKleene}
If $R$ is a quasiorder on a non-empty set $U$, then $\mathbb{RS}$ is a quasi-Nelson algebra.
\end{lemma}

\begin{proof}
Since $RS$ is a completely distributive lattice, we have only to show that $\mathbb{RS}$ is
a Kleene algebra. Let $x = (X^\DOWN,X^\UP)$ and $y = (Y^\DOWN,Y^\UP)$. Then,
\begin{align*}
x \wedge c(x) &= (X^\DOWN \cap (X^\UP)^\prime , X^\UP \cap (X^\DOWN)^\prime ) = (\emptyset,X^\UP \setminus X^\DOWN); \\
y \vee c(y) &= (Y^\DOWN \cup (Y^\UP)^\prime, Y^\UP \cup  (Y^\DOWN)^\prime ) = ((Y^\UP \setminus Y^\DOWN)',U).
\end{align*}
Hence, the condition $x \wedge c(x) \leq y \vee c(y)$ is satisfied.
\end{proof}

Our next lemma presents some properties of the completely join-irreducible elements of $RS$.
The set $\mathcal{J}$ is defined as in \eqref{Eq:RS-JoinIrr}.

\begin{lemma} \label{Lem:RS-Star}
If $R$ is a quasiorder on  a non-empty set $U$, then the following assertions hold:
\begin{enumerate}[\rm (a)]
\item $\mathcal{J}^- = \{(\emptyset,\{x\}^\UP) \mid \ |R(x)| \geq 2 \}$;
\item $(\emptyset,\{x\}^\UP)^* = (R(x),R(x)^\UP)$ for all $(\emptyset,\{x\}^\UP) \in \mathcal{J}^-$;
\item $\mathcal{J}^+ = \{( R(x), R(x)^\UP) \mid \ |R(x)| \geq 2 \}$;
\item $\mathcal{J}^* = \{( \{x\}, \{x\}^\UP) \mid \ R(x) = \{x\} \, \}$.
\end{enumerate}
\end{lemma}

\begin{proof} (a) Let $j = (X^\DOWN,X^\UP) \in \mathcal{J}^-$. Then, 
$j \leq c(j)$ and $j = j \wedge c(j) = (\emptyset,X^\UP \setminus X^\DOWN)$. This implies
$j = (\emptyset,\{x\}^\UP)$ for some $x$ such that $|R(x)| \geq 2$. Conversely, if
$j = (\emptyset,\{x\}^\UP)$, then $c(j) = (U \setminus \{x\}^\UP,U)$ implying
$j \leq c(j)$ and $j \in \mathcal{J}^-$.

(b) Let $x \in U$ be such that $|R(x)| \geq 2$. Then,
$j = (\emptyset,\{x\}^\UP)  \in \mathcal{J}^-$ and by Lemma~\ref{Lem:KleeneDivision}(c),
$j^* \in \mathcal{J}^+$ and therefore we must have $j^* = (R(y),R(y)^\UP)$ for some $y \in U$.
Because $j^*$ is the least element not included in $c(j) = (U \setminus \{x\}^\UP,U)$,
we have $R(y) \nsubseteq U \setminus \{x\}^\UP$, that is, 
$R(y) \cap \{x\}^\UP \neq \emptyset$. Thus, there exists $z \in U$
such that $y \, R \, z$ and $z \, R \, x$ implying $y \, R \, x$.
This gives $R(x) \subseteq R(y)$ and $R(x)^\UP \subseteq R(y)^\UP$.
Hence,  $(R(x),R(x)^\UP) \leq (R(y),R(y)^\UP)=j^*$. As 
$(R(x),R(x)^\UP)  \nleq c(j)$, we get $(R(x),R(x)^\UP)=j^*$.

(c) is now obvious by (a), (b), and Lemma~\ref{Lem:KleeneDivision}. 

(d) If $j = j^*$, then by \eqref{Eq:RS-JoinIrr}, (a), and (c) we must have 
$j = (R(x),R(x)^\UP) = (\{x\},\{x\}^\UP)$ for some $x \in U$
such that $R(x) = \{x\}$.

Conversely, suppose $j = (\{x\},\{x\}^\UP)$ for some $x$ such that 
$R(x) = \{x\}$. Then, $j \notin \mathcal{J}^-$
and, by Lemma~\ref{Lem:KleeneDivision}(a), either $j^* \in \mathcal{J}^-$
or $j^* \in \mathcal{J}^*$. Clearly, $j^* \in \mathcal{J}^-$ is not possible,
since $R(x) = \{x\}$. 
\end{proof}

\begin{proposition} \label{Prop:RoughNelson}
If $R$ is a quasiorder on a non-empty set $U$, then $\mathbb{RS}$ is a Nelson algebra
such that the underlying lattice $RS$ is algebraic.
\end{proposition}

\begin{proof}
By Theorem~\ref{Thm:Interpolation} and  Proposition~\ref{Prop:CharactNelson},
we have to show only that the quasi-Nelson algebra $\mathbb{RS}$ 
has property (M). Let $p,q \in \mathcal{J}$, where $\mathcal{J}$ is defined
as in \eqref{Eq:RS-JoinIrr}, and suppose that 
\[
p^*,q^* \leq p,q.  
\]
We will show that in this case there exists an element $k \in \mathcal{J}$
such that 
\[
p^*,q^* \leq k \leq p,q.  
\]
We may exclude the cases $p = p^*$ and $q = q^*$, because they
imply directly $k = p$ or  $k = q$.

Now, $p^*, q^* < p,q$ implies $p,q \in \mathcal{J}^+$. 
Hence,  $p = (R(x),R(x)^\UP)$ and $q = (R(y),R(y)^\UP)$
for some elements $x,y\in U$ such that $|R(x)|,|R(y)| \geq 2$.
By Lemma~\ref{Lem:RS-Star}(b), 
$p^* =(\emptyset ,\{x\}^\UP)$ and $q^* =(\emptyset ,\{y\}^\UP)$.
Then, $p^* \leq q$ gives $x \in \{x\}^\UP \subseteq R(y)^\UP$.
Hence, there exists an element $z\in U$ such that $y \, R \,z$ and 
$x \, R \, z$. We have to consider two cases: (i) $|R(z)| \geq 2$ and 
(ii) $R(z) =\{z\}$.

(i) Assume that $|R(z)| \geq 2$. Then clearly, 
$k=(\emptyset ,\{z\}^\UP)$ is a completely join-irreducible
element. Observe that  $x, y \in \{z\}^\UP$
implies $\{x\}^\UP, \{y\}^\UP \subseteq \{z\}^\UP$, 
whence we obtain $p^* \leq k$ and $q^* \leq k$. Since $z \in R(x)$ and 
$z \in R(y)$, we get also $\{z\}^\UP \subseteq R(x)^\UP$ and
$\{z\}^\UP \subseteq R(y)^\UP$ implying $k \leq p$ and $k \leq q$.

(ii) Suppose $R(z) = \{z\}$. Then $k =(\{z\},\{z\}^\UP)$ is a completely 
join-irreducible element. Because  $x \, R \,z$ and 
$y \, R \, z$, we obtain
\begin{align*}
 p^* = (\emptyset ,\{x\}^\UP) & \leq (\{z\},\{z\}^\UP)=k, \\
 q^* = (\emptyset,\{y\}^\UP) & \leq (\{z\},\{z\}^\UP)=k, \\ 
 k = (\{z\},\{z\}^\UP) & \leq (R(x),R(x)^\UP) = p, \mbox{ and} \\
 k = (\{z\},\{z\}^\UP) & \leq (R(y),R(y)^\UP) = q.
\end{align*}
Hence, $p^{\ast },q^{\ast }\leq k\leq p,q$ is satisfied in both cases (i) and (ii).
\end{proof}

\begin{example} \label{Exa:Chain}
For any binary relation $R$ on $U$, a set $C \subseteq U$ is called a \emph{connected component}
of $R$, if $C$ is an equivalence class of the smallest equivalence relation 
containing $R$.
In \cite{JRV09}, we presented a decomposition theorem stating that for any
left-total relation,
\[ RS \cong \prod_{C \in \mathfrak{Co}} RS(C),\]
where $\mathfrak{Co}$ is the set of connected components of $R$ and $RS(C)$
is the set of rough sets on the component $C$ determined by the restriction
of $R$ to the set $C$. Note that $R$ is said to be \emph{left-total} if for any $x$,
there exists $y$ such that $x \, R \, y$.

For any equivalence, the connected components are just equivalence classes. If an
equivalence class $C$ consists of a single element, say $a$, then
$RS(C) = \{ (\emptyset,\emptyset), (\{a\}, \{a\})$, and  
$RS(C) = \{ (\emptyset,\emptyset), (\emptyset,C), (C,C) \}$ in
case $|C| \geq 2$.
This then means that $RS$ is isomorphic to the direct product of chains
of two and three elements.

In case of quasiorders, the height of components cannot be limited.
Let us consider the following simple case. Assume that
$U = \{1,2,\ldots,n\}$ is a set of $n$ consecutive natural numbers and
consider its usual order $\leq$.
For any $X \subseteq U$, $X^\UP = \{1,\ldots,i\}$, where
$i$ is the maximal element of $X$ and $X^\DOWN = \{j,\ldots,n\}$,
where $j$ is the least $j$ such that $\{j,\ldots,n\} \subseteq X$. 
All elements of $U$ belong to the same component and
$RS$ has the members: 
\begin{align*}
\mathcal{A}(\{1,\ldots,n\}) & = (\{1,\ldots,n\},U); \\
                           & \ \, \vdots \\
\mathcal{A}(\{n-1,n\}) & = (\{n-1,n\},U); \\
\mathcal{A}(\{n\}) & = (\{n\},U); \\
\mathcal{A}(\{n-1\}) & = (\emptyset,\{1,\ldots,n-1\}); \\
                     & \ \, \vdots \\
\mathcal{A}(\{1\}) & = (\emptyset,\{1\}); \\
\mathcal{A}(\emptyset) & = (\emptyset,\emptyset). 
\end{align*}
Thus, $RS$ forms a chain of $2n$ elements.
\end{example}

We conclude this section by considering three-valued {\L}ukasiewicz algebras.
It is known that three-valued {\L}ukasiewicz algebras coincide with the
semisimple Nelson algebras --- see \cite{Cign07}, where further references 
can be found. 

Following A.~Monteiro \cite{Mont63}, we can define a \emph{three-valued 
{\L}ukasiewicz algebra} as an algebra $(A,\vee,\wedge,c,\Delta,1)$ such that 
$A$ is a distributive lattice and the following equations
are satisfied:
\begin{enumerate}[({\L}1)]
\item $c(c(x)) = x$;
\item $c(x \wedge y) = c(x) \vee c(y)$;
\item $c(x) \vee \Delta(x) = 1$;
\item $x \wedge c(x) = c(x) \wedge \Delta(x)$;
\item $\Delta(x \wedge y) = \Delta(x) \wedge \Delta(y)$.
\end{enumerate}

\begin{proposition} \label{Prop:Luka}
If $R$ is a quasiorder, then the rough set lattice $\mathbb{RS}$ is a 
three-valued {\L}ukasiewicz algebra if and only if $R$ is an equivalence.
\end{proposition}

\begin{proof} 
If $R$ is an equivalence, then $\mathbb{RS}$ is a three-valued {\L}ukasiewicz algebra 
such that $\Delta(\mathcal{A}(X)) = (X^\UP,X^\UP)$ as shown by P.~Pagliani \cite{Pagliani97,PagChak08}.

Conversely, suppose by contradiction
that $\mathbb{RS}$ is a three-valued {\L}ukasiewicz algebra, but the
quasiorder $R$ is not symmetric. Then, there
exist $x,y \in U$ such that $(x,y) \in R$, but $(y,x) \notin R$. 
Now $|R(x)| \geq 2$, $j = (\emptyset,\{x\}^\UP) \in \mathcal{J}^-$,
and $j \leq c(j)$. For $j$, there exists the element $\Delta(j)$ satisfying ({\L}4), that is,
$c(j) \wedge \Delta(j) = j \wedge c(j) = j$. From this we obtain
\[
(U \setminus \{x\}^\UP, U) \wedge \Delta(j) = (\emptyset,\{x\}^\UP).
\]
This means that $\Delta(j) = (Y^\DOWN,Y^\UP)$, where $Y^\UP = \{x\}^\UP$ and
$Y^\DOWN \subseteq \{x\}^\UP$. Assume that $Y^\DOWN \ne \emptyset$. Then,
there exists $z \in Y^\DOWN \subseteq \{x\}^\UP$. 
So, $z \, R \,x$, and now $x \,  R \, y$ implies $z \, R \, y$, that is, $y \in R(z)$.
Since $z \in Y^\DOWN$, we have $y \in R(z) \subseteq Y \subseteq Y^\UP = \{x\}^\UP$.
But $y \in \{x\}^\UP$ is not possible, since $(y,x) \notin R$. 
So, we must have  $Y^\DOWN = \emptyset$ and $\Delta(j) = (\emptyset, \{x\}^\UP) = j$. Then,
\[ c(j) \vee \Delta(j) = c(j) \vee j = c(j) = (U \setminus \{x\}^\UP,U) \ne (U,U).\]
This contradicts ({\L}3) and therefore $\mathbb{RS}$ is not a three-valued {\L}ukasiewicz algebra.
\end{proof}

\section{Proof of the Representation Theorem} \label{Sec:Representation}

Next, we present the proof of  Theorem~\ref{Thm:MAIN}.
Let $\mathbb{A}$ be a Nelson algebra such that its underlying lattice $A$
is algebraic. We denote by $\mathcal{J}$ the set of
completely join-irreducible elements of $A$. We define a mapping 
$\rho \colon \mathcal{J} \to \mathcal{J}^- \cup \mathcal{J}^*$.
For all $j \in \mathcal{J}$, let
\[
\rho(j) = \left \{
\begin{array}{ll}
j    & \mbox{ if $j \in \mathcal{J}^- \cup \mathcal{J}^*$,}\\
j^*  & \mbox{ otherwise.}
\end{array}
\right .
\]
Notice that $\rho(\rho(j)) = \rho(j)$ and 
$\rho(j) = \rho(j^*)$ for all $j \in \mathcal{J}$.

\bigskip

\paragraph{\it Proof of Theorem~\ref{Thm:MAIN}} \
Let us set $U = \mathcal{J}$ and define a binary relation
$R$ on $U$ by
\[ 
  x \, R \, y \iff \rho(x) \leq \rho(y).
\]
Then, $R$ is a quasiorder on $U$. Observe
that $x \, R \, x^*$ and $x^* \, R \, x$ for all $x \in U$.
Thus, $R(x) = R(x^*)$ and $\{x\}^\UP = \{x^*\}^\UP$ for every $x \in U$.
If $x \in \mathcal{J}^- \cup \mathcal{J}^+$, then $|R(x)| \geq 2$,
because $x \ne x^*$. If $x \in \mathcal{J}^*$, then $R(x) = \{x\}$.
Namely, if $x \in \mathcal{J}^*$ and $x \, R \, y$, then
$x = \rho(x) \leq \rho(y)$. Now $\rho(y)^* \leq x^* = x \leq \rho(y)$
implies $\rho(y) \in \mathcal{J}^*$ since $\rho(y) \in \mathcal{J}^+$ is 
not possible. So, $\rho(y)  = \rho(y)^*$, which
implies $y \in \mathcal{J}^*$ by the definition of $\rho$.
Now, $x \leq y$ and $y \leq x$ yield $y = x$.

The algebra $\mathbb{RS}$ of rough sets determined by the relation $R$ is a
Nelson algebra such that the underlying lattice $RS$ is an algebraic lattice
by Proposition~\ref{Prop:RoughNelson}. Let $\mathcal{J}(RS)$
denote the sets of completely join-irreducible elements of $RS$.
We show first that $\mathcal{J}$ and $\mathcal{J}(RS)$ are order-isomorphic.

Let us define a mapping $\varphi \colon \mathcal{J} \to \mathcal{J}(RS)$
by setting
\[
\varphi(x) = \left \{
\begin{array}{ll}
(\emptyset,\{x\}^\UP) & \mbox{ if $x \in \mathcal{J}^-$;} \\
(R(x),R(x)^\UP)       & \mbox{ otherwise.}
\end{array}
\right .
\]
The map $\varphi$ is well-defined. Namely, if $x \in \mathcal{J}^-$, then
$|R(x)| \geq 2$ and 
$$\varphi(x) = (\emptyset,\{x\}^\UP)\in \mathcal{J}(RS)^-.$$ 
If $x \in \mathcal{J}^+$, then also $|R(x)| \geq 2$, and 
$$\varphi(x) = (R(x),R(x)^\UP) \in \mathcal{J}(RS)^+.$$ 
For any $x \in \mathcal{J}^*$, $R(x) = \{x\}$ gives 
$$\varphi(x) = (R(x),R(x)^\UP) = (\{x\},\{x\}^\UP) \in \mathcal{J}(RS)^*$$ (cf. Lemma~\ref{Lem:RS-Star}).

We show that $\varphi$ is an order-embedding. The proof is divided into four cases:
\begin{enumerate}[\rm (i)]
\item $x \in \mathcal{J}^-$ and $y \in \mathcal{J}^-$;
\item $x \in \mathcal{J}^-$ and $y \notin \mathcal{J}^-$;
\item $x \notin \mathcal{J}^-$ and $y \notin \mathcal{J}^-$;
\item $x \notin \mathcal{J}^-$ and $y \in \mathcal{J}^-$.
\end{enumerate}

(i) Let $x, y \in \mathcal{J}^-$. Then, $\rho(x) = x$ and $\rho(y) = y$.
Now $x \leq y$ implies $\rho(x) \leq \rho(y)$ and $x \, R \, y$. So,
$\{x\}^\UP \subseteq \{y\}^\UP$ and $\varphi(x) = (\emptyset,\{x\}^\UP) \leq
(\emptyset,\{y\}^\UP) ) = \varphi(y)$.
Conversely, if $\varphi(x) \leq \varphi(y)$, then $x \in \{x\}^\UP \subseteq
\{y\}^\UP$ and $x \, R \, y$. Thus, $x = \rho(x) \leq \rho(y) = y$.

(ii) Let $x \in \mathcal{J}^-$ and $y \notin \mathcal{J}^-$. Then,
$\rho(x) = x$ and $\rho(y) = y^*$. If $x \leq y$, then
$y^* \leq x^*$ and we have $x,y^* \leq x^*$ and
$x,y^* \leq y^{**} = y$. Thus, $x^{**}, y^* \leq x^*,y$,
which implies by Proposition~\ref{Prop:CharactNelson} that
there exists $z \in \mathcal{J}$ such that
$x, y^* \leq z,z^* \leq x^*,y$. This then gives
$\rho(x), \rho(y) \leq \rho(z)$ and we obtain
$x \, R \, z$ and $y \, R \, z$. Therefore, $z \in R(y)$, $x \in R(y)^\UP$,
$\{x\}^\UP \subseteq R(y)^{\UP \UP} = R(y)^\UP$, and
$\varphi(x) = (\emptyset,\{x\}^\UP) \leq (R(y),R(y)^\UP) = \varphi(y)$.

On the other hand, if $\varphi(x) \leq \varphi(y)$, then
$x \in \{x\}^\UP \subseteq R(y)^\UP$ and so there exists $z$
such that $x \, R \, z$ and $y \, R \, z$. Thus,
$x = \rho(x) \leq \rho(z)$ and $y ^* = \rho(y) \leq \rho(z)$. 
If $\rho(z) = z$, then $z \in \mathcal{J}^- \cup \mathcal{J}^*$
and $z \leq z^*$. We have $x \leq z$ and $y^* \leq z \leq z^*$
implying $x \leq z \leq y$. If $\rho(z) = z^*  \neq z$, then
$z \in \mathcal{J}^+$ and $z^* < z$. Then $x \leq z^* < z$
and $y^* \leq z^*$ giving $x < z \leq y$.

(iii) Now $x,y \notin \mathcal{J}^-$, $\rho(x) = x^*$, and
$\rho(y) = y^*$. If $x \leq y$, then $\rho(y) = y^* \leq
x^* = \rho(x)$ and $y \, R \, x$. Therefore, $R(x) \subseteq R(y)$,
and $\varphi(x) = (R(x),R(x)^\UP) \leq (R(y),R(y)^\UP) = \varphi(y)$.
Conversely, if $\varphi(x) \leq \varphi(y)$, then $R(x) \subseteq R(y)$.
Since $x \in R(x)$, we have $y \, R \, x$. Thus,
$y^* = \rho(y) \leq \rho(x) = x^*$ and $x \leq y$.

(iv) If $x \notin \mathcal{J}^-$ and $y \in \mathcal{J}^-$, then
$x \leq y$ is impossible, because $\mathcal{J}^-$ is a down-set.
Similarly, $\varphi(x) = (R(x),R(x)^\UP) \leq (\emptyset,\{y\}^\UP)
= \varphi(y)$ is not possible, because $R(x) \neq \emptyset$. 

\medskip
We have shown that $x \leq y$ if and only if $\varphi(x) \leq \varphi(y)$
in all cases. Hence, $\varphi \colon \mathcal{J} \to \mathcal{J}(RS)$
is an order-embedding. Next we prove that $\varphi$ is onto $\mathcal{J}(RS)$.
Let $j \in \mathcal{J}(RS)$. We consider three disjoint cases.

(i) Assume $j = (\emptyset,\{x\}^\UP)$ for some $x \in U = \mathcal{J}$
such that $|R(x)| \geq 2$. If $x \in \mathcal{J}^-$, then
$\varphi(x) = (\emptyset,\{x\}^\UP) = j$. 
If $x \in \mathcal{J}^+$,
then $\varphi(x) = (R(x),R(x)^\UP)$ and
$\varphi(x^*) = (\emptyset,\{x^*\}^\UP) =  (\emptyset,\{x\}^\UP) = j$.
Case $x \in \mathcal{J}^*$ may be excluded, since it means
$R(x) = \{x\}$.

(ii) If $j = (\{x\},\{x\}^\UP)$ for some $x \in U$ such
that $R(x) = \{x\}$, then $ x \, R \, x^*$ and $x^*  R \, x$
imply $x = x^*$. Thus, $x \in \mathcal{J}^*$, and we infer $\varphi(x) = j$.

(iii) Suppose that $j = (R(x),R(x)^\UP)$ for some $x \in U$
such that $|R(x)| \geq 2$. Similarly as in case (i), 
if $x \in \mathcal{J}^+$, then $\varphi(x) = j$ and
if $x \in \mathcal{J}^-$, then $\varphi(x^*) = j$. 
Case $x \in \mathcal{J}^*$ is not possible.

Thus, $\varphi$ is an order-isomorphism of the ordered set $\mathcal{J}$ 
onto the ordered set $\mathcal{J}(RS)$. 
Because $A$ and $RS$ are isomorphic as lattices, they are isomorphic 
as Heyting algebras also.  
To show the isomorphism of the algebras $\mathbb{A}$ and $\mathbb{RS}$, 
by Corollary~\ref{Cor:DemorgranIsom} it is sufficient to prove that 
$\varphi(j^*) = \varphi(j)^*$ for all $j \in \mathcal{J}$.

\medskip\noindent%
We consider three cases:

(i) If $x \in \mathcal{J}^-$, then $x^* \in \mathcal{J}^+$ and
\[\varphi(x^*) = (R(x^*),(R(x^*)^\UP) = (R(x),R(x)^\UP) = 
(\emptyset,\{x\}^\UP)^* = \varphi(x)^*.\]

(ii) If $x \in \mathcal{J}^*$, then $x = x^*$ and
\[
\varphi(x^*) = (\{x^*\},\{x^*\}^\UP) =  (\{x\},\{x\}^\UP) = \varphi(x) = \varphi(x)^* .
\]

(iii) If $x \in \mathcal{J}^+$, then $x^* \in \mathcal{J}^-$ and
\[
\varphi(x^*) = (\emptyset,\{x^*\}^\UP) = (\emptyset,\{x\}^\UP)
= (R(x),R(x)^\UP)^* = \varphi(x)^*.
\]
Thus, $\mathbb{A}$ and $\mathbb{RS}$ are isomorphic Nelson algebras.
\hfill \qed

\medskip\medskip%

\begin{example}\label{Exa:Construction}
Let us consider the Nelson algebra $\mathbb{A}$ presented in Figure~\ref{Fig:Fig1}(a).
The operation $c$ is defined by $c(0) = 1$, $c(a) = f$, $c(b) = e$, and $c(c) = d$.
The set of completely join-irreducible elements $\mathcal{J}$ is $\{a,b,d,e,f\}$.
It is easy to observe that $a^* = e$, $b^* = f$, and $d^* = d$. Hence,
$\mathcal{J}^- = \{a,b\}$, $\mathcal{J}^* = \{d\}$, and $\mathcal{J}^+ = \{e,f\}$.
Let us consider the construction of the relation $R$ presented in the proof of Theorem~\ref{Thm:MAIN}.

\begin{figure}[ht] \label{Fig:Fig1}
\centering
\includegraphics[width=80mm]{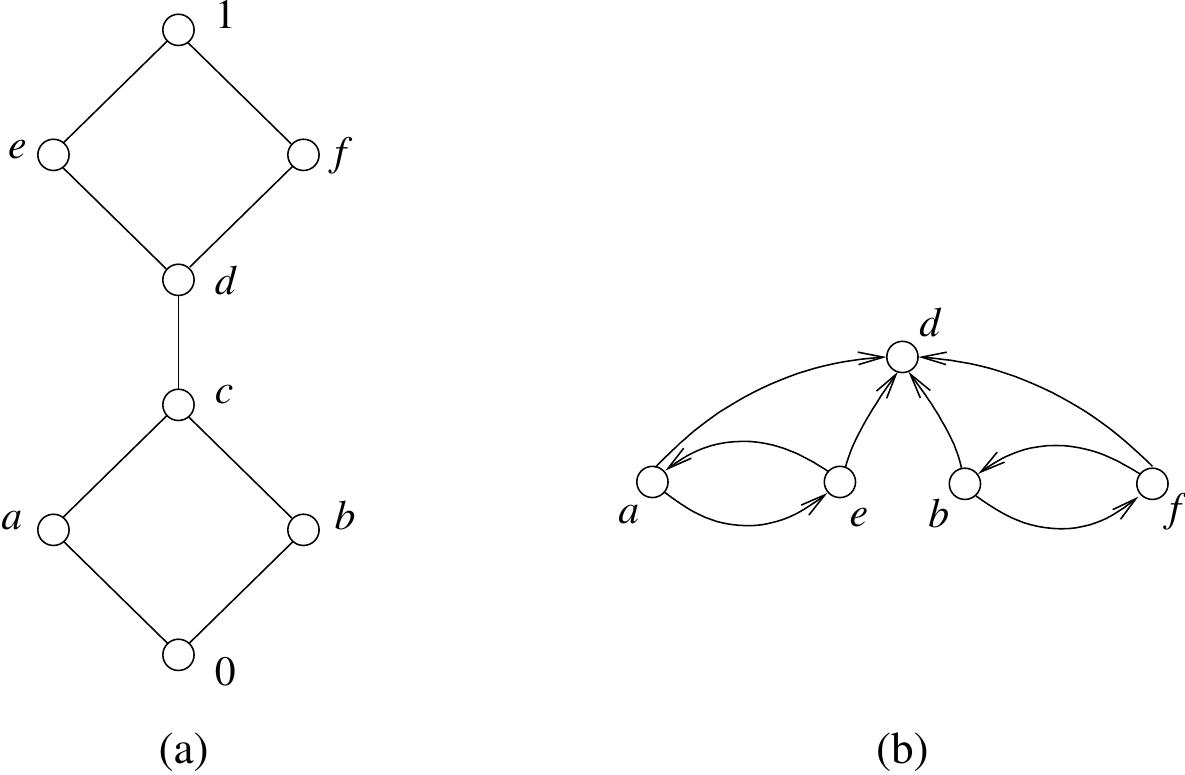}
\caption{ }
\end{figure}

The map $\rho$ is defined as $\rho(a) = \rho(e) = a$, $\rho(d) = d$, and
$\rho(b) = \rho(f) = b$. The relation $R$ on $U = \mathcal{J}$ is then defined
by $x \, R \, y$ iff $\rho(x) \leq \rho(y)$ and is depicted in 
Figure~\ref{Fig:Fig1}(b). Note that since $R$ is reflexive, there should be
an arrow from each circle to the circle itself, but such loops are omitted.
Now $RS = \{ (\emptyset,\emptyset)$, $(\emptyset,\{a,e\})$, $(\emptyset,\{b,f\})$, $(\emptyset,\{a, b,e, f\})$,
$(\{d\}, U)$, $(\{a,d,e\}, U)$, $(\{b,d,f\}, U)$, $(U,U)\}$ and clearly $\mathbb{RS} \cong \mathbb{A}$.
\end{example}

For rough sets lattices determined by equivalences, there exists the following representation
theorem: for every regular double Stone lattice $(A,\vee,\wedge, {^*}, {^+}, 0, 1)$, there exists a 
set $U$ and an equivalence $R$ on $U$ such that $A$ is isomorphic to a
subalgebra of $(RS, \cup, \cap, {^*}, {^+}, (\emptyset,\emptyset), (U,U))$; note that in $RS$, 
the pseudocomplement of $\mathcal{A}(X)$ is $(U \setminus X^\UP, U \setminus X^\UP)$ 
and its dual pseudocomplement is  $(U \setminus X^\DOWN, U \setminus X^\DOWN)$; see \cite{DemOrl02, GeWa92}. 
By applying  Theorem~\ref{Thm:MAIN} and Proposition~\ref{Prop:Luka}, we can prove the 
following  result for rough sets determined by equivalences.

\begin{corollary} \label{Cor:RepresentationEq}
Let\/ $\mathbb{A}$ be a semisimple Nelson algebra defined on an algebraic lattice. 
Then, there exists a set $U$ and an equivalence $R$ on $U$ such that $\mathbb{A} \cong \mathbb{RS}$.
\end{corollary}

\begin{proof} Suppose that $\mathbb{A}$ is a semisimple Nelson algebra.
Then, by Theorem~\ref{Thm:MAIN}, there exists a set $U$
and a quasiorder $R$ on $U$ such that $\mathbb{A}$ and $\mathbb{RS}$ are 
isomorphic  Nelson algebras. Because $\mathbb{A}$ is a semisimple Nelson algebra,
then $\mathbb{A}$ and $\mathbb{RS}$ are isomorphic three-valued  
{\L}ukasiewicz algebras. This implies by  Proposition~\ref{Prop:Luka} that 
$R$ must be an equivalence.
\end{proof}

Let us denote by $R_\mathcal{J}$ the quasiorder on 
$\mathcal{J}$ constructed in the proof of Theorem~\ref{Thm:MAIN}. 
Let us note that the set $\mathcal{J}$ and 
the relation $R_\mathcal{J}$ are not necessary minimal in the 
sense that there may exist a set $U$ of smaller
cardinality than $\mathcal{J}$ and a quasiorder $R$ on $U$ 
determining the same Nelson algebra. For instance, 
in case of Example~\ref{Exa:Construction}, the same algebra can be obtained
by the relation $R = \delta_U \cup \{ (a,c), (b,c)\}$ on
the three-element set $U = \{a,b,c\}$, where
$\delta_U$ denotes the identity relation of $U$.
However, $R_\mathcal{J}$ has the property that any equivalence
class of $R_\mathcal{J} \cap {R_{\mathcal{J}}}^{-1}$ 
is of the form $\{j,j^*\}$, where $j \in \mathcal{J}$.
Therefore, any rough set algebra $\mathbb{RS}$ determined by a quasiorder
can be generated also by the quasiorder $R_\mathcal{J}$ having the
property that equivalence classes of $R_\mathcal{J} \cap {R_{\mathcal{J}}}^{-1}$
have at most two elements.
In case of three-valued {\L}ukasiewicz algebras this then means
that each equivalence class of $R_\mathcal{J}$ has at most two elements.
Hence, the construction is minimal in the above sense.

\begin{remark} \label{Rem:Boolean} 
Notice that the Nelson algebra  $\mathbb{A}$  is isomorphic to a Boolean
algebra whenever the quasiorder $R_\mathcal{J}$ of our construction is a partial 
order; in fact, for our construction, the following statements are 
equivalent: 
\begin{enumerate}[\rm (a)]
\item $R_\mathcal{J}$ is a partial order;
\item $R_\mathcal{J} = \delta_\mathcal{J}$;
\item $\mathbb{RS} \cong (\wp(\mathcal{J}), \cup, \cap,{}', \emptyset,\mathcal{J})$;
\item $\mathbb{A}$ is a Boolean algebra.
\end{enumerate}

Namely, if $R_\mathcal{J}$ is a partial order, then our construction
gives $j = j^*$ for all $j \in \mathcal{J}$. So, by Lemma~\ref{Lem:StarProperties},
$i \leq j$ implies $i = j$ for all $i,j \in \mathcal{J}$. 
Thus, $(\mathcal{J},\leq)$ is an antichain and $R_\mathcal{J} = \delta_\mathcal{J}$. 
As (b)$\Rightarrow$(a) is obvious, we get (a)$\Leftrightarrow$(b).

If $R_\mathcal{J} = \delta_\mathcal{J}$, then $RS = \{(X,X) \mid X \subseteq \mathcal{J}\}$ 
and $c(X,X)=(\mathcal{J} \setminus X, \mathcal{J} \setminus X)$.
Therefore, we have  $\mathbb{RS} \cong (\wp(\mathcal{J}), \cup, \cap,{}', \emptyset,\mathcal{J})$.
Thus, (b)$\Rightarrow$(c), and (c)$\Rightarrow$(d) is clear by Theorem~~\ref{Thm:MAIN}.
If $\mathbb{A}$ is defined on an algebraic Boolean lattice, then this lattice is 
atomistic and $\mathcal{J}$ coincides with the set of its atoms. This means that $(\mathcal{J},\leq)$ is 
an antichain and $R_\mathcal{J} = \delta_\mathcal{J}$ by our construction.
So, (d)$\Rightarrow$(b).
\end{remark}

\providecommand{\bysame}{\leavevmode\hbox to3em{\hrulefill}\thinspace}
\providecommand{\MR}{\relax\ifhmode\unskip\space\fi MR }
\providecommand{\MRhref}[2]{%
  \href{http://www.ams.org/mathscinet-getitem?mr=#1}{#2}
}
\providecommand{\href}[2]{#2}


\end{document}